\documentclass{amsart}
\usepackage{fullpage}\usepackage{pictex}
\usepackage{amsmath}
\usepackage{amssymb}
\usepackage{amsfonts}
\usepackage{amsopn}
\usepackage{epsfig}
\usepackage{amsfonts}
\usepackage{latexsym}
\usepackage{bm}
\usepackage{stmaryrd}
\usepackage{amsthm}
\usepackage{mathrsfs}

\newtheorem{theorem}{Theorem}[section]

\newtheorem{proposition}[theorem]{Proposition}

\newtheorem{remark}{Remark}

\newcommand{\R}{\mathbb R} \newcommand{\Z}{\mathbb Z}
\newcommand{\N}{\mathbb{N}} \newcommand{\Q}{\mathbb{Q}}

\usepackage{tikz}

\title{Sharpening  Vahlen's result in Diophantine approximation}

\author{Ayreena Bakhtawar
\and
Cor Kraaikamp
}
\newcommand{\Addresses}{{
  \bigskip
  \footnotesize

A.~Bakhtawar, \textsc{Centro di Ricerca Matematica Ennio De Giorgi, Scuola Normale Superiore, Piazza dei Cavalieri 3,
56126 Pisa, Italy and \\
Institute of Mathematics, Polish Academy of Sciences, ul.  Sniadeckich 8, 00-656
Warszawa, Poland}\par\nopagebreak
  \textit{E-mail address}: \texttt{ayreena.bakhtawar@sns.it,abakhtawar@impan.pl}

  \medskip

C. ~Kraaikamp\textsc{
Delft University of Technology, EWI (DIAM), Mekelweg 4, 2628 CD Delft, the Netherlands} 
\par\nopagebreak
\textit{E-mail address}: \texttt{c.kraaikamp@tudelft.nl}
}
}

\begin{document}

\begin{abstract}  \noindent In this paper we refine Vahlen's 1895 result in Diophantine approximation by providing sharper bounds for the approximation coefficients, especially when at least one of the partial quotients $a_n$ or $a_{n+1}$ of the regular continued fraction expansion $[a_0;a_1,a_2,\dots]$ of $x$ is 1. An improvement of Vahlen's result was already given in papers by Jaroslav Han\u{c}l (\cite{[H]}), Han\u{c}l and Silvie Bahnerova (\cite{[HB]}), and by Dinesh Sharma Bhattarai (\cite{[B]}), but the approach of the present paper is very different from Han\u{c}l c.s. We believe that the geometrical methods used in this paper not only offer a significant improvement over Vahlen's result, but also yield new insights that can contribute to improving Borel's classical constant.
\end{abstract}

\maketitle

\section{Introduction}\label{section1}

One of the classic results in Diophantine approximation is Vahlen's well-known theorem from 1895, \cite{[V]}. Let\footnote{Without loss of generality we will assume in this paper that $x\in [0,1).$} $x\in\R$, with \emph{regular continued fraction} (RCF) expansion $x=[a_0;a_1,a_2,\dots]$, where $a_0\in\Z$, such that $x-a_0\in [0,1)$, and\footnote{Recall that the RCF-expansion of $x\in\R$ is finite if and only if $x\in\Q$. In that case $x$ has two RCF-expansions. In all other cases the RCF-expansion of $x\in\R$ is unique.}  $a_i\in\N$, for $n\geq 1$. Furthermore, by taking finite truncation, let the sequence of \emph{regular continued fraction convergents} of $x$ be given by $(p_n/q_n)$; i.e., $p_n/q_n=[a_0;a_1,\dots,a_n]$ and $\text{gcd}\{ p_n,q_n\}=1$ for $n\geq 1$. Vahlen's theorem states that for $x\in\R$ and $n\geq 1,$
$$
\min \left\{ q_{n-1}^2\left| x-\frac{p_{n-1}}{q_{n-1}}\right| ,\,\, q_n^2\left| x-\frac{p_n}{q_n}\right| \right\} < \frac{1}{2}.
$$
In view of this result, as well as similar results on three consecutive convergents by Borel, Bagemihl \& McLaughlin, Tong and (many) others (see~\cite{[Bor1], [Bor2], [BMc], [DK], [T]}), we define the so-called \emph{approximation coefficients} $\Theta_n=\Theta_n(x)$ by
$$
\Theta_n(x) = q_n^2\left| x-\frac{p_n}{q_n}\right| ,\quad \text{for $n\geq 1$}.
$$
In 1903,  \'Emile Borel \cite{[Bor1], [Bor2]} showed, that if $x\in\R$, with three consecutive convergents $\tfrac{p_{n-1}}{q_{n-1}}$, $\tfrac{p_n}{q_n}$
and $\tfrac{p_{n+1}}{q_{n+1}}$, one has that
\begin{equation}\label{Borel}
\min \left\{ \Theta_{n-1}(x),\Theta_n(x),\Theta_{n+1}(x)\right\} < \frac{1}{\sqrt{5}}.
\end{equation}
To give these results by Vahlen and Borel (and later ones like~\eqref{improvementBorel}) some urgency we recall here a result by Legendre from 1798 (see~\cite{[L]}, and also~\cite{[BJ]}, for a more recent proof and generalization of Legendre's result); let  $A,B\in\Z$ with $B>0$, $\text{gcd}\{ A,B\}=1$ be such, that 
\begin{equation}\label{legendre}
B^2\left| x-\frac{A}{B}\right| < \frac{1}{2},
\end{equation}
then the rational number $A/B$ is an RCF-convergent of $x$. I.e., there exists an $n\in\N$, such that $A=p_n$ and $B=q_n$. So Legendre's result states that \emph{if} we want to approximate an irrational number $x$ well by a rational number $A/B$ (such that~\eqref{legendre} holds), this rational number is an RCF-convergent of $x$. Due to Vahlen's result we know that at least one of each two consecutive RCF-convergents of $x$ approximates $x$ well in this sense, and Borel's result shows that at least one in three consecutive RCF-convergents approximates $x$ even better. In fact, in~\cite{[BJ]} it is shown, that if
$$
B^2\left| x-\frac{A}{B}\right| < 1,
$$
we have that $A/B$ is either an RCF-convergent, or a so-called first of last mediant.

Only as recent as 2015, Vahlen's result has been sharpened by Jaroslav Han\u{c}l in~\cite{[H]}; at least one of two consecutive convergents $\tfrac{p_{n-1}}{q_{n-1}}$, $\tfrac{p_n}{q_n}$ satisfies the inequality
\begin{equation}\label{Hancl2015}
\left| x-\frac{p}{q}\right| < \left( 2 + \frac{2(q-1)}{q^2(q+1)}\right)^{-1} \cdot\frac{1}{q^2}.
\end{equation} 
The term $H(q):= \left( 2 + \frac{2(q-1)}{q^2(q+1)}\right)^{-1}$ in~\eqref{Hancl2015} has been further sharpened in 2021 by Han\u{c}l and Silvie Bahnerova in~\cite{[HB]}, and by Dinesh Sharma Bhattarai in 2023; ~\cite{[B]}.\smallskip

More recently, Borel's result was sharpened by Han\u{c}l and Kit Nair in~\cite{[HN]} who showed that if we replace the $<$ sign in~\eqref{Borel} with $\leq,$  the constant $\sqrt{5}$ can be replaced by $\sqrt{5}+\frac{4-5\sqrt{5}+\sqrt{61}}{2q^2}.$ In this paper we will mainly focus on sharpening the Vahlen result.

As the sequence $(q_n)_{n\geq1 }$ is rapidly growing (at least as fast as the \emph{Fibonacci numbers}, which we would get if $x$ is the golden mean $g=(\sqrt{5}-1)/2$), as a result $H(q)$ tends to $\tfrac{1}{2}$ quite quickly. For example, when $q=610$ (the 15th Fibonacci number), we have $H(q) =  0.4999987.$ Therefore for large $n,$ the improvement over Vahlen's result quickly diminishes, and the bound given in~\cite{[H]} along with improvements in~\cite{[HB],[B]} might be considered somewhat conservative, particularly when $a_{n+1}\geq 2.$ For further details see Sections~\ref{section2} and~\ref{section3}.
This rapid growth of $(q_n)_{n\geq1 }$ motivates us to explore a different sharpening of Vahlen's result, which is discussed in the next two sections. Unlike the previous approaches, this sharpening will not depend on the size of $q_{n}$, but rather on the local size of the partial quotients $a_{n+2}$  and $a_{n-1}$. By ``local size," we mean that it does not depend on the index $n$ or, more specifically, the sizes of $q_{n-1}$ or $q_{n}$.

In Section~\ref{section2}, we will sharpen the Vahlen's result whenever the pair of partial quotients $(a_n,a_{n+1})\neq (1,1)$. It should be mentioned that our approach in Section~\ref{section2} is to some extend implicit in the proofs of the results in~\cite{[H],[HB],[B]}, the difference being that we consider the values of two consecutive approximation coefficients simultaneously. In Section~\ref{section3}, we will tackle the more challenging case $(a_n,a_{n+1}) = (1,1)$.  We will show that improvement over Vahlen's result is still possible as soon as one also knows the value of at least one of the partial quotients $a_{n+2}$ and $a_{n-1}$.

\section{Using the Ito-Nakada-Tanaka natural extension of the RCF}\label{section2}
In 1981, Hitoshi Nakada (in~\cite{[N]}), and Shunji Ito and Shigeru Tanaka (in~\cite{[TI]}) introduced and studied their respective classes of $\alpha$-expansions. For $\alpha =1$ (and $\alpha=\tfrac{1}{2}$), these two classes coincide: $\alpha =1$ corresponds to the RCF, (and $\alpha =\tfrac{1}{2}$ corresponds to the \emph{nearest integer continued fraction} (NICF) expansion). Nakada obtained the \emph{natural extension} of the dynamical systems underlying his $\alpha$-expansions for all $\alpha\in [\tfrac{1}{2},1]$ (which also includes RCF case), derived the invariant measures, and showed that these dynamical systems are ergodic. In fact, he obtained a much stronger mixing property. Nakada's natural extension for the RCF was fundamental in the proof of the so-called \emph{Doeblin-Lenstra} conjecture in~\cite{[BJW]}, which described the distribution of the sequence of approximation coefficients 
$\Theta_n(x)$ of $x$ for almost every $x\in\R$ 
and played a key-role in the metrical theory of continued fractions in the last decades; see e.g.\ Chapter 4 in~\cite{[IK]}.\smallskip\

\subsection{Nakada's natural extension for the RCF}\label{subsectionNakada} Let $\mathcal{T}:[0,1)\times [0,1]$ be defined by
\begin{equation}\label{naturalextensionmap}
\mathcal{T}(x,y) = \left( T(x),\frac{1}{\left\lfloor \tfrac{1}{x}\right\rfloor + y}\right),\quad \text{if $(x,y)\in (0,1)\times [0,1]$},
\end{equation}
and $\mathcal{T}(0,y) = (0,y)$, for $y\in [0,1]$. Here $T:[0,1)\to[0,1)$ is the so-called \emph{Gauss map}, defined by $T(0)=0$, and
$$
T(x) = \frac{1}{x}-\left\lfloor \frac{1}{x}\right\rfloor ,\quad \text{for $x\in (0,1)$}.
$$
For $x\in (0,1)$ and $n\geq 1$, with RCF-convergents $(\tfrac{p_n}{q_n})$,  we define 
$$
t_n=t_n(x) := T^n(x)\quad \text{as \emph{the future of $x$ at time $n$}},
$$
and 
$$
v_n=v_n(x) := \frac{q_{n-1}}{q_n}\quad \text{as \emph{the past of $x$ at time $n$}}.
$$
Note that if $x=[0;a_1,a_2,\dots]$, we have that $t_n=[0;a_{n+1},a_{n+2},\dots],$ and from the well-known recurrence relations for the $q_n$ we furthermore have that $v_n=[0;a_n,a_{n-1}\dots,a_1]$; see, e.g.~pages 7--10 in~\cite{[Sch]}. It follows from the definition of $\mathcal{T}$ in~\eqref{naturalextensionmap} that:
$$
\mathcal{T}^n(x,0) = (t_n,v_n),\quad \text{for $n\geq 0$}.
$$
A classical result is now that:
\begin{equation}\label{thetan}
\Theta_n(x) = \frac{t_n}{1+t_nv_n},\quad \text{for $n\geq 1$};
\end{equation}
see e.g.~\cite{[DK], [IK], [K]}, but also (2) in~\cite{[B]} and (3.2) in~\cite{[H]}. In fact, from~\eqref{thetan} and the definition of the Gauss map $T$, one easily finds that:
\begin{equation}\label{thetan-1}
\Theta_{n-1}(x) = \frac{v_n}{1+t_nv_n},\quad \text{for $n\geq 1$}.
\end{equation}
In view of~\eqref{thetan} and~\eqref{thetan-1}, Henk Jager and the second author introduced in~\cite{[JK]} the map $\Psi: \Omega := [0,1]\times [0,1]\to \R^2$ defined by
$$
\Psi (t,v) = \left( \frac{v}{1+tv},\frac{t}{1+tv}\right).
$$
In~\cite{[JK]}, it is shown that $\Psi (\Omega)$ is the  triangle $\Delta$ in $\R^2$ with vertices $(0,0)$, $(0,1)$, and $(1,0)$. Since $(\Theta_{n-1},\Theta_n) = \Psi (t_n,v_n)$, for $n\geq 1$, we find that $\Theta_{n-1} + \Theta_n < 1$, which implies Vahlen's result: 
$$\min \{ \Theta_{n-1},\Theta_n\} < \tfrac{1}{2}.$$

Furthermore, apart from a set of Lebesgue measure 0, the map $\Psi :\Omega\to\Delta$ is bijective. If we define the map $F:\Delta\to\Delta$ by $F=\Psi\circ \mathcal{T}\circ \Psi^{-1}$, we find for two consecutive $\Theta_{n-1}$ and $\Theta_n$ that
$$
F(\Theta_{n-1},\Theta_n) = \left( \Theta_n,\Theta_{n+1}\right) ,\quad \text{for $n\geq 1$}.
$$
As a result, one can derive the formula of  W.B.~Jurkat and A.~Peyerimhoff (\cite{[JP]}):
\begin{equation}\label{thetasA}
\Theta_{n+1} = \Theta_{n-1} +a_{n+1}\sqrt{1-4\Theta_{n-1}\Theta_n} - a_{n+1}^2\Theta_n,
\end{equation}
but also, 
$$ 
\Theta_{n-1} = \Theta_{n+1} + a_{n+1}\sqrt{1-4\Theta_n\Theta_{n+1}} - a_{n+1}^2\Theta_n.
$$ 
See also~\cite{[JK]}, where~\eqref{thetasA} was used to obtain, in a simple way, a generalization of Borel's celebrated result, previously obtained among others by N.~Obrechkoff (\cite{[O]}) and by F.~Bagemihl and J.R.\ McLaughlin (\cite{[BMc]}):
\begin{equation}\label{improvementBorel}
\min \{ \Theta_{n-1},\Theta_n,\Theta_{n+1}\} \leq \frac{1}{\sqrt{a_{n+1}^2+4}}\leq \frac{1}{\sqrt{5}},\quad \text{for $n\geq 1$},
\end{equation}
but also a 1983 result by Jingcheng Tong (\cite{[T]}):
$$
\max \{ \Theta_{n-1},\Theta_n,\Theta_{n+1}\} \geq \frac{1}{\sqrt{a_{n+1}^2+4}},\quad \text{for $n\geq 1$}.
$$

\subsection{The ``easy case'' $(a_n,a_{n+1})\neq (1,1)$}\label{subsec:valensconstants}
Now, in view of~\eqref{thetan} and~\eqref{thetan-1}, we define for $a\in\N$ the sets $V_a$ and $H_a$ in $\Omega$ where the partial quotients $a_{n+1}$ resp.\ $a_n$ are constant whenever $(t_n,v_n)\in V_a$ resp.\ $(t_n,v_n)\in H_a$:
$$
V_a = \left\{ (x,y)\in\Omega\,\, \Big{|}\,\, \frac{1}{a+1} < x\leq \frac{1}{a}\right\} ,
$$
as shown in Figure~\ref{figure1}, and 
$$
H_a = \left\{ (x,y)\in\Omega\,\, \Big{|}\,\, \frac{1}{a+1} < y\leq \frac{1}{a}\right\},
$$
One can easily show that $\mathcal{T}(V_a)=H_a$ for all $a\in\N$. Additionally the following conditions hold,
\begin{eqnarray*}
\mathcal{T}^n(x,y)\in V_a &\Leftrightarrow & a_{n+1}=a,\,\, n\geq 0;\\
\mathcal{T}^n(x,y)\in H_a &\Leftrightarrow & a_n=a,\,\, n\geq 1.
\end{eqnarray*}

Next, note that if $c\in [0,1]$ is a constant, then $\Psi$ maps the vertical line segment $t=c$, $v\in [0,1]$, to the intersection of the line $\beta = -c^2\alpha +c$ with $\Delta$, and $\Psi$ maps the horizontal line segment $v=c$, $t\in [0,1]$, to the intersection of the line $\beta = -\tfrac{1}{c^2}\alpha +\tfrac{1}{c}$ with $\Delta$. Due to this, for $a\geq 2,$ we have that $\Psi (V_a)$ is a quadrangle in $\Delta,$ with vertices 
$$
\Psi \left(\tfrac{1}{a},0\right) = \left( 0,\tfrac{1}{a}\right),\quad \Psi \left( \tfrac{1}{a+1},0\right) = \left( 0,\tfrac{1}{a+1}\right),\quad \Psi \left( \tfrac{1}{a+1},1\right) = \left( \tfrac{a+1}{a+2},\tfrac{1}{a+2}\right),
$$
and ${\displaystyle \Psi \left( \tfrac{1}{a},1\right) = \left( \tfrac{a}{a+1},\tfrac{1}{a+1}\right)}$. For $a=1$, we have that $\Psi (V_a)$ is a triangle in $\Delta$, with vertices 
$$
\left( 0,\tfrac{1}{2}\right),\quad \left( \tfrac{2}{3},\tfrac{1}{3}\right),\quad \text{and $(0,1)$}.
$$
\begin{figure}[h]
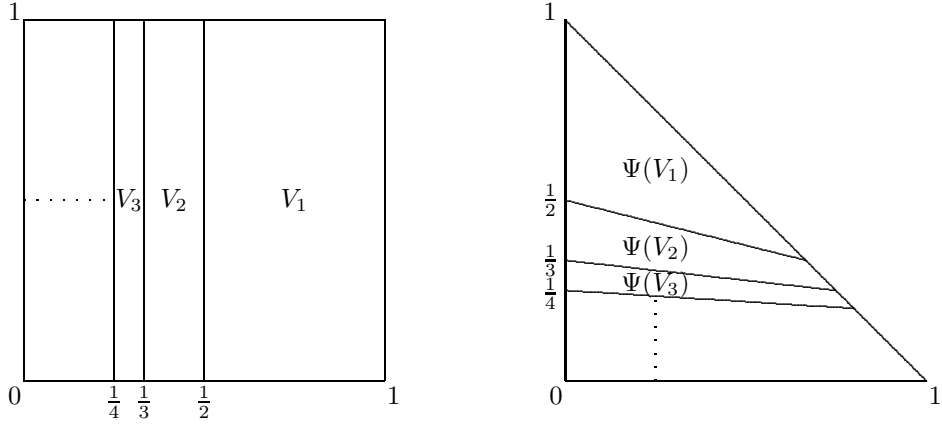

$$
\beginpicture
    \setcoordinatesystem units <0.4cm,0.4cm>
    \setplotarea x from -15 to 15, y from 0 to 12
    \putrule from 3 0 to 3 12
    \putrule from 3 0 to 15 0
    \put {$0$} at 2.7 -0.5
    \put {$0$} at -15.3 -0.5
    \put {$1$} at 15.3 -0.5
    \put {$1$} at -2.7 -0.5
    \put {$1$} at 2.5 12.3
    \put {$1$} at -15.3 12.3
    \putrule from -15 0 to -3 0
    \putrule from -15 0 to -15 12
    \putrule from -3 0 to -3 12
    \putrule from -15 12 to -3 12
    \putrule from -9 0 to -9 12
\putrule from -11 0 to -11 12
\putrule from -12 0 to -12 12

    \put {$\tfrac{1}{2}$} at 2.5 6
    \put {$\tfrac{1}{2}$} at -9 -0.7
    \put {$\tfrac{1}{3}$} at 2.5 4
    \put {$\tfrac{1}{3}$} at -11 -0.7
    \put {$\tfrac{1}{4}$} at 2.5 2.9
    \put {$\tfrac{1}{4}$} at -12 -0.7
    \put {$\Psi(V_1)$} at 6 7
    \put {$\Psi(V_2)$} at 6 4.4
    \put {$\Psi(V_3)$} at 6 3.2
    \put {$V_1$} at -6 6
    \put {$V_2$} at -10 6
    \put {$V_3$} at -11.5 6 

\setlinear \plot
3 12 15 0
/
\setlinear \plot
3 6 11 4
/
\setlinear \plot
3 4 12 3
/
\setlinear \plot
3 3 12.6 2.4
/

\setdots
\putrule from 6 0 to 6 3
\putrule from -15 6 to -12 6

\endpicture
$$ \caption[triangle]{The natural extension $\Omega$ of the RCF (left), and its image $\Delta$ under $\Psi$ (right)}\label{figure1}
\end{figure}
Thus $\Psi (V_a)$ is bounded by the lines: $\alpha = 0$, $\alpha+\beta =1$, $\beta = -\frac{\alpha}{a^2}+\frac{1}{a}$ (this is the `top-line' of $\Psi (V_a)$), and $\beta = -\frac{\alpha}{(a+1)^2}+\frac{1}{a+1}$ (this is the `bottom-line' of $\Psi (V_a)$, and the `top-line' of $\Psi (V_{a+1})$); see Figure~\ref{figure1}. Also note, that $\Psi (H_a)$ is the reflection of $\Psi (V_a)$ in the line $\beta = \alpha$. So $\Psi (H_a)$ is bounded by the lines $\beta =0$, $\alpha +\beta =1$, the ``right-hand line'' $\beta = -a^2\alpha + a$, and ``left-hand line'' $\beta = -(a+1)^2\alpha+(a+1)$.\medskip\

Now, let $m=\min \{ a_n,a_{n+1}\}$ and $M=\max \{ a_n,a_{n+1}\}$. In the next section, i.e., Section \ref{section3},  we will consider the more difficult case where $M=1$ (so when both $a_n$ and $a_{n+1}$ are equal to $1$). Here, we consider results that improve upon Vahlen's result in the case where $M\geq 2$. We have the following results for this case.

\begin{theorem}\label{theoremeasycase}
Let $x=[0;a_1,a_2,\dots, a_n,a_{n+1},\dots]$, $m=\min \{ a_n,a_{n+1}\}\geq 1$, and $M=\max \{ a_n,a_{n+1}\}\geq 2$. We consider the following cases.
\begin{itemize}
\item[($i$)]
Let $m=M\geq 2$, then we have that:
$$
\min\{ \Theta_{n-1}(x),\Theta_n(x)\} \leq \tfrac{m}{m^2+1}\leq \tfrac{2}{5},
$$
and
$$
\max\{ \Theta_{n-1}(x),\Theta_n(x)\} \geq \tfrac{m+1}{(m+1)^2+1}.
$$

\item[($ii$)]
Let $1\leq m < M$, then we have that:
$$
\min\{ \Theta_{n-1}(x),\Theta_n(x)\} \leq \tfrac{m+1}{(m+1)M+1} \leq \tfrac{2}{5},
$$
and
$$
\max \{ \Theta_{n-1}(x),\Theta_n(x)\} \geq \tfrac{M}{(m+1)M+1}.
$$
\end{itemize}
\end{theorem}

\begin{remark}\label{remarkVahlen}{\rm
Note that the results of case~($i$) are also valid if $M=1.$ However, in this case, there is no improvement of the Vahlen bound $\tfrac{1}{2}$.\hfill $\triangle$}
\end{remark}

\begin{proof}
In case $m=M\geq 2$, we observe that $\Psi (V_m)\cap \Psi (H_m)$ is a quadrangle in $\Delta$ with vertices:
$$
\Psi \left( \tfrac{1}{m},\tfrac{1}{m} \right) = \left( \tfrac{m}{m^2+1},\tfrac{m}{m^2+1}\right),\,\, \Psi \left( \tfrac{1}{m},\tfrac{1}{m+1} \right) = \left( \tfrac{m}{m^2+m+1}, \tfrac{m+1}{m^2+m+1}\right), \,\, \Psi \left( \tfrac{1}{m+1},\tfrac{1}{m} \right) = \left( \tfrac{m+1}{m(m+1)+1}, \tfrac{m}{m(m+1)+1}\right),
$$
and ${\displaystyle \Psi \left( \tfrac{1}{m+1},\tfrac{1}{m+1} \right) = \left( \tfrac{m+1}{(m+1)^2+1},\tfrac{m+1}{(m+1)^2+1}\right)}$. Now, one easily observes that $\Psi (V_m)\cap \Psi (H_m)$ is symmetric in the diagonal $\beta = \alpha$, and that on $\Psi (V_m)\cap \Psi (H_m),$ the function $f(\alpha,\beta)=\alpha+\beta$ attains the maximum at the point ${\displaystyle \left( \tfrac{m}{m^2+1},\tfrac{m}{m^2+1}\right)}$, the maximum value being $\tfrac{2m}{m^2+1}$.
But then we immediately find, that for $(\Theta_{n-1}(x),\Theta_n(x))\in \Psi (V_m)\cap \Psi (H_m)$ we have that:
$$
\min \{ \Theta_{n-1}(x),\Theta_n(x)\} \leq \tfrac{m}{m^2+1}\leq \tfrac{2}{5}.
$$
Similarly, we see that on $\Psi (V_m)\cap \Psi (H_m)$ the function $f(\alpha,\beta)=\alpha+\beta$ attains the minimum at the point  ${\displaystyle \left( \tfrac{m+1}{(m+1)^2+1},\tfrac{m+1}{(m+1)^2+1}\right)}$, the minimum value being $\tfrac{2(m+1)}{(m+1)^2+1}$. Thus, we obtain:
$$
\max\{ \Theta_{n-1}(x),\Theta_n(x)\} \geq \tfrac{m+1}{(m+1)^2+1}.
$$
This completes the proof of  case ($i$). \medskip\

The proof of case ($ii$) is similar to the proof of case ($i$); the only difference is that the calculations are slightly nastier, as $1\leq m<M$. Now we must consider either $\Psi (V_m)\cap \Psi (H_M)$ or $\Psi (V_M)\cap \Psi (H_m)$, but by symmetry we only need to consider one of these two cases, say, $\Psi (V_M)\cap \Psi (H_m)$. Of course, a quick \& dirty upper bound for the Vahlen constant can be found by realizing that $(\Theta_{n-1}(x),\Theta_n(x))\in \Psi (V_M)\cap \Psi (H_m)$ implies that $(\Theta_{n-1}(x),\Theta_n(x))\in \Psi (V_M)$, and therefore we find $$\min \{ \Theta_{n-1}(x),\Theta_n(x)\} < \tfrac{M}{M^2+1}\leq \tfrac{2}{5}.$$ In this case, $\Psi (V_M)\cap \Psi (H_m)$ is again a quadrangle, bounded by the lines: $\beta=-(m+1)^2\alpha+(m+1)$, $\beta=-m^2\alpha+m$ (the ``left-'' resp.\ ``right-hand'' line), $\beta = -\frac{\alpha}{(M+1)^2} +\frac{1}{M+1}$, and $\beta = -\frac{\alpha}{M^2}+\frac{1}{M}$ (the ``bottom-'' resp.\ ``top-line''). This quadrangle has vertices
\begin{eqnarray*}
\Psi \left( \tfrac{1}{M},\tfrac{1}{m+1}\right) = \left( \tfrac{M}{(m+1)M+1}, \tfrac{m+1}{(m+1)M+1}\right), && \Psi \left( \tfrac{1}{M},\tfrac{1}{m}\right) = \left( \tfrac{M}{mM+1},\tfrac{m}{mM+1}\right),\\
\Psi \left( \tfrac{1}{M+1},\tfrac{1}{m+1}\right) = \left( \tfrac{M+1}{(m+1)(M+1)+1}, \tfrac{m+1}{(m+1)(M+1)+1}\right), &&
\Psi \left( \tfrac{1}{M+1},\tfrac{1}{m}\right) = \left( \tfrac{M+1}{m(M+1)+1},\tfrac{m}{m(M+1)+1}\right) .
\end{eqnarray*}
These are resp.\ the top-left, top-right, bottom-left and bottom-right vertex of $I_{M,m}=\Psi (V_M)\cap \Psi (H_m)$; see also Figures~\ref{figure21} and~\ref{figure2}.
\begin{figure}[h]
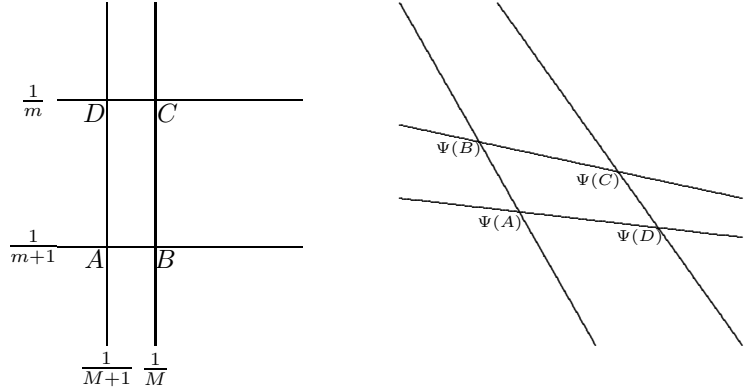

$$
\beginpicture
    \setcoordinatesystem units <0.65cm,0.65cm>
    \setplotarea x from 0 to 12, y from 0 to 6
    \putrule from 0 2 to 5 2
    \putrule from 0 5 to 5 5
    \putrule from 1 0 to 1 7
    \putrule from 2 0 to 2 7
    \put {$\tfrac{1}{M+1}$} at 1 -0.5
    \put {$\tfrac{1}{M}$} at 2 -0.5
    \put {$\tfrac{1}{m+1}$} at -0.5 2
    \put {$\tfrac{1}{m}$} at -0.5 5
    \put {$A$} at 0.75 1.75
    \put {$B$} at 2.2 1.75
    \put {$C$} at 2.25 4.75
    \put {$D$} at 0.75 4.75
    \put {\tiny$\Psi (A)$} at 9.05 2.5
    \put {\tiny$\Psi (B)$} at 8.2 4
    \put {\tiny$\Psi (C)$} at 11.05 3.35
    \put {\tiny$\Psi (D)$} at 11.9 2.2

\setlinear \plot
9 7 14 0
/
\setlinear \plot
7 7 11 0
/
\setlinear \plot
7 4.5 14 3
/
\setlinear \plot
7 3 14 2.2
/

\endpicture
$$ \caption[triangle]{Left: $V_M\cap H_m$ in $\Omega$. Right:  $I_{M,m}=\Psi(V_M\cap H_m)$ in $\Delta$} \label{figure21}
\end{figure}

Since $M>m\geq 1$, we observe from the vertices that $\Psi (V_M)\cap \Psi (H_m)$ ``lies below'' the diagonal $\beta =\alpha$. Thus, at the point ${\displaystyle \left( \tfrac{M}{(m+1)M+1}, \tfrac{m+1}{(m+1)M+1}\right)}$, which is the top left-hand vertex of $I_{M,m}$,  the $y$-value represents the largest value $\Theta_n(x)$ can attain on $\Psi (V_M)\cap \Psi (H_m)$. Hence we have: 
$$
\min \{ \Theta_{n-1}(x),\Theta_n(x)\} \leq \tfrac{m+1}{(m+1)M+1} \leq \tfrac{2}{5}.
$$
Similarly, at the same point ${\displaystyle \left( \tfrac{M}{(m+1)M+1}, \tfrac{m+1}{(m+1)M+1}\right)}$ the $x$-value represents the smallest value $\Theta_{n-1}(x)$ can attain on $\Psi (V_M)\cap \Psi (H_m)$. Thus, since $\Theta_{n-1}(x)>\Theta_n(x)$ on $\Psi (V_M)\cap \Psi (H_m)$, we have
$$
\max \{ \Theta_{n-1}(x),\Theta_n(x)\} \geq \tfrac{M}{(m+1)M+1}.
$$
\end{proof}

\begin{figure}[h]
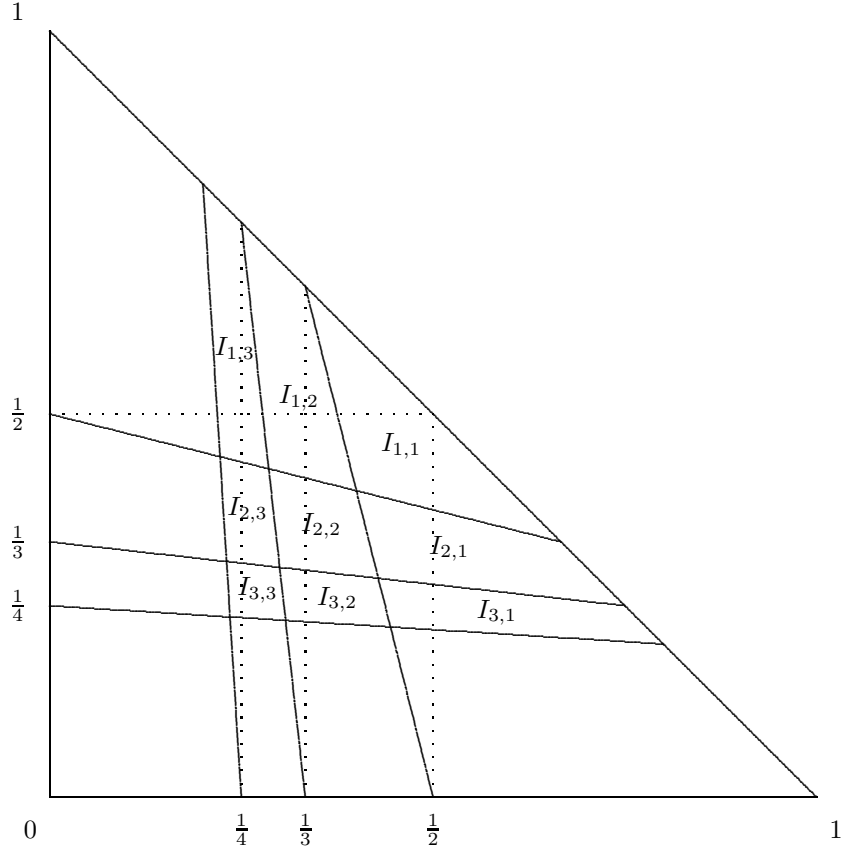

$$
\beginpicture
    \setcoordinatesystem units <0.85cm,0.85cm>
    \setplotarea x from 0 to 12, y from 0 to 10
    \putrule from 0 0 to 0 12
    \putrule from 0 0 to 12 0
    \put {$0$} at -0.3 -0.5
    \put {$1$} at 12.3 -0.5
    \put {$1$} at -0.5 12.3
    \put {$I_{1,1}$} at 5.5 5.5
    \put {$I_{2,2}$} at 4.25 4.25
    \put {$I_{3,2}$} at 4.5 3.1
    \put {$I_{2,3}$} at 3.1 4.5
    \put {$I_{3,3}$} at 3.25 3.25
    \put {$I_{2,1}$} at 6.25 3.9
    \put {$I_{3,1}$} at 7 2.9
    \put {$I_{1,3}$} at 2.9 7
    \put {$I_{1,2}$} at 3.9 6.25

    \put {$\tfrac{1}{2}$} at -0.5 6
    \put {$\tfrac{1}{2}$} at 6 -0.5
    \put {$\tfrac{1}{3}$} at -0.5 4
    \put {$\tfrac{1}{3}$} at 4 -0.5
    \put {$\tfrac{1}{4}$} at -0.5 3
    \put {$\tfrac{1}{4}$} at 3 -0.5

\setlinear \plot
0 12 12 0
/
\setlinear \plot
0 6 8 4
/
\setlinear \plot
6 0 4 8
/
\setlinear \plot
0 4 9 3
/
\setlinear \plot
4 0 3 9
/
\setlinear \plot
0 3 9.6 2.4
/
\setlinear \plot
3 0 2.4 9.6
/

\setdots
\putrule from 0 6 to 6 6
\putrule from 6 0 to 6 6
\putrule from 4 0 to 4 8
\putrule from 3 0 to 3 9
\endpicture
$$ \caption[triangle]{Intersections $I_{i,j}=\Psi(V_i\cap H_j)$ in $\Delta$, for $i,j=1,2,3$} \label{figure2}
\end{figure}

\begin{remark}\label{remarkVahlen2}{\rm
To determine the quadrangles $I_{M,m} = \Psi \left( V_m\cap H_m\right) = \Psi (V_m)\cap \Psi (H_m)$ the calculations from part ($ii$) suffice; just substitute $m=M$ in the above vertices in case ($ii$) to find the vertices from case ($i$). However, since in case ($i$) the quadrangle $I_{M,m}$ is on the diagonal $\beta = \alpha$, while in case ($ii$) the quadrangle $I_{M,m}$ is ``under'' $\beta = \alpha$, and $I_{m,M}$ is ``above'' $\beta = \alpha$, the vertices of these quadrangles yielding $\min \{ \Theta_{n-1},\Theta_n\}$ and $\max \{ \Theta_{n-1},\Theta_n\}$ on these quadrangles would be different, and we still would need to consider two different cases.\hfill $\triangle$}
\end{remark}

\section{The ``difficult case'' $(a_n,a_{n+1}) = (1,1)$}\label{section3}
In Subsection~\ref{subsec:valensconstants}, we postponed the case where $a_n=1=a_{n+1}$ to this section. To motivate this, recall from~\cite{[J1],[J2]} that the dynamical system $(\Delta, \mu,F)$, with $\mu$  a measure on $\Delta$, defined by
$$
\mu (E) = \frac{1}{\log 2}\int_E \frac{{\rm d}\alpha\, {\rm d}\beta}{\sqrt{1-4\alpha\beta}},\quad E\subseteq \Delta,\,\, \text{$E$ Borel measurable},
$$
forms an ergodic system. Obviously, problems arise when both $\Theta_{n-1}(x)$ and $\Theta_n(x)$ are close to $\tfrac{1}{2}$. By the definition of $\Psi,$ this occurs precisely  when both $t_n$ and $v_n$ are close to 1. To investigate this deeper, let $d\in\N$, and let $x_d$ be the rational number, given by
$$
x=\frac{1}{\displaystyle d+\frac{1}{\displaystyle 1 +\frac{1}{\displaystyle 1+\frac{1}{\displaystyle d}}}};
$$
i.e., $x_d=[0;d,1,1,d]$. Using the well-known recursion relations for $p_n$ and $q_n$, or through simple calculation, one finds
$$
\frac{p_1}{q_1}=\frac{1}{d},\quad \frac{p_2}{q_2} = \frac{1}{d+1},\quad \frac{p_3}{q_3}=\frac{2}{2d+1},\quad \frac{p_4}{q_4}=x_d=\frac{2d+1}{2d^2+2d+1}.
$$
In the context of the previous section, we have that: $t_2=[0;1,d]=v_2$, and thus we immediately obtain
$$
\Theta_1=\Theta_2 = \frac{t_2}{1+t_2v_2} = \frac{\frac{d}{d+1}}{1+\left(\frac{d}{d+1}\right)^2} = \frac{d(d+1)}{d^2+(d+1)^2} = \frac{d(d+1)}{2d^2+2d+1}.
$$
This result can also be derived directly  from the definitions of $\Theta_1$ and $\Theta_2$:
$$
\Theta_1 = q_1^2\left| x_d-\frac{p_1}{q_1}\right| = d^2\left| \frac{2d+1}{2d^2+2d+1}-\frac{1}{d}\right| = \frac{d(d+1)}{2d^2+2d+1},
$$
and similarly for $\Theta_2$, where we use that $q_2=d+1$. For example, when $d=100$ we have $\Theta_1=\Theta_2=0.4999752$, while in this case, one easily finds that $H(q_1) = 0.499951$ and $H(q_2)=0.499952$. Furthermore,
$$
\Theta_3=q_3^2\left| x_d-\frac{p_3}{q_3}\right| = (2d+1)^2\left| \frac{2d+1}{2d^2+2d+1}-\frac{2}{2d+1}\right| = \frac{2d+1}{2d^2+2d+1}.
$$
So for $d=100,$ one finds $\Theta_3=\frac{201}{20201}=0.009950002$. Obviously, since $x_d=\frac{p_4}{q_4}$, we have that $\Theta_4=0$. This leads to the following proposition, which demonstrates that the example with $d=100$ is not a coincidence.

\begin{proposition}\label{prop:counterexample}
Let $d\in\N$, $d\geq 4$, and let $x=[0;d,1,1,d]$. Then we have that:
$$
H(q_1) \leq \Theta_1\quad \text{ and }\quad H(q_2) \leq \Theta_2.
$$
\end{proposition}

\begin{proof}
We will only provide a proof of the second statement, as it is slightly more involved than the first one. Note that $q=q_2=d+1.$ Thus, we need to show that
$$
\left( 2 +\frac{2(q-1)}{q^2(q+1)}\right)^{-1} \leq\,\, \frac{q(q-1)}{2(q-1)^2+2(q-1)+1},
$$
which is equivalent to
$$
\frac{2(q-1)((q-1)+1)+1}{(q-1)q} \leq 2 +\frac{2(q-1)}{q^2(q+1)}.
$$
Simplifying this expression, we get
$$
\frac{1}{q-1} \leq \frac{2(q-1)}{q(q+1)},
$$
which can further shown to be equivalent to $5q\leq q^2+2$, which holds as $q=d+1$ and we assumed that $d\geq 4$.
\end{proof}

In order to improve upon Vahlen's result in the case where $a_n=1=a_{n+1}$, it is obvious from Theorem~\ref{theoremeasycase}, Remark~\ref{remarkVahlen}, and Proposition~\ref{prop:counterexample} that in this case also the values of $a_{n-1}$ and/or $a_{n+2}$ should be taken into consideration; we need to refine the triangle $I_{1,1}=\Psi (V_1\cap H_1)$. Let $m=\min \{ a_{n-1},a_{n+2}\}$ and $M=\max \{ a_{n-1},a_{n+2}\}$. Note that the interval $(\tfrac{1}{2},1)$ can be partitioned as follows: let $\Delta_{1,k} = \{ x\in (\tfrac{1}{2},1)\, |\, a_1(x)=1,\, a_2(x)=k\}$, for $k\in\N$, then $\Delta_{1,1} = \Big( \tfrac{1}{2}, \tfrac{2}{3}\Big)$, and $\Delta_{1,k} = \Big[ \tfrac{k}{k+1}, \tfrac{k+1}{k+2}\Big)$, for $k\geq 2$, and we see that:
$$
\left( \tfrac{1}{2},1\right) = \bigcup_{k=1}^{\infty} \Delta_{1,k} = \Big( \tfrac{1}{2}, \tfrac{2}{3}\Big) \cup \bigcup_{k=2}^{\infty} \Big[ \tfrac{k}{k+1}, \tfrac{k+1}{k+2}\Big) ,
$$
Furthermore, for $k\in\N$ define:
$$
V_{1,k} = \Delta_{1,k}\times [0,1]\quad \text{and}\,\, H_{1,k} = [0,1]\times \Delta_{1,k}.
$$ 

Since $a_n=1=a_{n+1}$, it follows that $(\Theta_{n-1}(x),\Theta_n(x))\in \Psi (V_{1,m}\cap\Psi (H_{1,M})$, or $(\Theta_{n-1}(x),\Theta_n(x))\in \Psi (V_{1,M}\cap\Psi (H_{1,m})$, depending on whether $a_{n-1}=m$ and $a_{n+2}=M$, or vice versa. As in the proof of Theorem~\ref{theoremeasycase}, due to symmetry (since reflecting $\Psi (V_{1,k})$ in the line $\beta=\alpha$ gives $\Psi (H_{1,k})$), we only need to determine the various subsets $\Psi (V_{1,k})$ of $\Delta$, and next $\Psi (V_{1,M})\cap \Psi (H_{1,m})$. Note that $\Psi (V_{1,k})$ is a quadrangle, with vertices
$$
\Psi \left( \tfrac{k}{k+1},0\right) = \left( 0,\tfrac{k}{k+1}\right),\,\, \Psi \left( \tfrac{k+1}{k+2},0\right) = \left( 0,\tfrac{k+1}{k+2}\right),\,\,
\Psi \left( \tfrac{k}{k+1},1\right) = \left( \tfrac{k+1}{2k+1},\tfrac{k}{2k+1}\right),
$$
and ${\displaystyle \Psi \left( \tfrac{k+1}{k+2},1\right) = \left( \tfrac{k+2}{2k+3},\tfrac{k+1}{2k+3}\right)}$; see also Figure~\ref{figure3}. Note that
$$
\frac{k+1}{2k+1}\downarrow \frac{1}{2}\quad \text{and that}\quad \frac{k}{2k+1}\uparrow \frac{1}{2},\quad \text{as $k\to \infty$}.
$$
\begin{figure}[h]
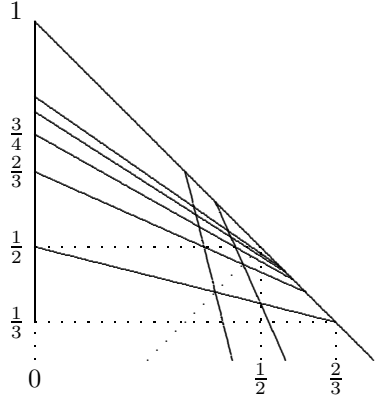

$$
\beginpicture
    \setcoordinatesystem units <0.5cm,0.5cm>
    \setplotarea x from 0 to 12, y from 0 to 9
    \putrule from 0 1 to 0 9
    
    \put {$1$} at -0.5 9.3
    \put {$0$} at 0 -0.5
    \put {$\tfrac{1}{2}$} at 6 -0.5
    \put {$\tfrac{2}{3}$} at 8 -0.5

    \put {$\tfrac{1}{2}$} at -0.5 3
    \put {$\tfrac{2}{3}$} at -0.5 5
    \put {$\tfrac{3}{4}$} at -0.5 6
    \put {$\tfrac{1}{3}$} at -0.5 1

\setlinear \plot
0 9 9 0
/
\setlinear \plot
0 3 8 1
/
\setlinear \plot
0 5 7.2 1.8
/
\setlinear \plot
0 6 6.85 2.14 
/
\setlinear \plot
0 6.6 6.666 2.33333 
/
\setlinear \plot
0 7 6.545454 2.454545 
/
\setlinear \plot
4 5 5.25 0
/
\setlinear \plot
4.8 4.2 6.666666 0
/

\setdots
\putrule from 0 1 to 8 1
\putrule from 0 3 to 6 3
\putrule from 6 0 to 6 3
\putrule from 0 0 to 0 1
\putrule from 8 0 to 8 1 
\setlinear \plot
3 0 6 3
/
\endpicture
$$ \caption[triangle]{$\Psi (V_{1,a})$ stacked ``on top of each other'' for $a=1,2,3,4$, and part of $H_{1,1}$} \label{figure3}
\end{figure}

A short calculation yields that the boundary lines of $\Psi (V_{1,k})$ are given by $\alpha = 0$ (on the ``left-hand side'') and $\beta = \alpha$ (on the ``right-hand side''), that the ``bottom line'' is $\beta = -\frac{k^2\alpha}{(k+1)^2} + \frac{k}{k+1}$, and the ``top line'' is $\beta = -\frac{(k+1)^2\alpha}{(k+2)^2} + \frac{k+1}{k+2}$ (so the same expression as the bottom line, but $k$ replaced by $k+1$); see again Figure~\ref{figure3}.\smallskip\

Note that if $a_{n+2}=k$, we can immediately give an improvement of Vahlen's constant, by calculating the point of intersection of the ``top line'' and the diagonal $\beta=\alpha$. This yields that in case when $a_n=1=a_{n+1}$ and $a_{n+2}=k$, with $k\in\N$, we have
$$
\min \{ \Theta_{n-1}(x),\Theta_n(x)\} < \frac{(k+1)(k+2)}{(k+1)^2+(k+2)^2} \uparrow \tfrac{1}{2},\,\,\, \text{when $k\to\infty$}.
$$
If we determine the point of intersection of the ``bottom line'' and $\beta = \alpha$, we find that in this case (i.e., when $a_n=1=a_{n+1}$ and $a_{n+2}=k$):
$$
\max \{ \Theta_{n-1}(x),\Theta_n(x)\} > \frac{k(k+1)}{k^2+(k+1)^2}.
$$
However, here we have only\footnote{Using only the partial quotient $a_{n-1}$ would have yielded the same result since the reflection of $\Psi (V_{1,k})$ in $\beta = \alpha$ is $\Psi (H_{1,k})$.} used the partial quotient $a_{n+2}$. We have the following result.

\begin{theorem}\label{theoremdifficultcase}
Let $x=[0;a_1,a_2,\dots, a_n,a_{n+1},a_{n+2},\dots]$, $(a_n,a_{n+1})=(1,1)$, $m=\min \{ a_{n-1}, a_{n+2}\}$, and $M=\max \{ a_{n-1}, a_{n+2}\}$. We consider the following cases.
\begin{itemize}
\item[($i$)]
Let $1\leq m=M$, then we have that:
$$
\min\{ \Theta_{n-1}(x),\Theta_n(x)\} \leq \tfrac{(m+1)(m+2)}{(m+1)^2+(m+2)^2} < \tfrac{1}{2},
$$
and
$$
\max\{ \Theta_{n-1}(x),\Theta_n(x)\} > \tfrac{m(m+1)}{m^2+(m+1)^2}\geq \tfrac{2}{5}.
$$

\item[($ii$)]
Let $1\leq m < M$, then we have that:
$$
\min\{ \Theta_{n-1}(x),\Theta_n(x)\} < \tfrac{(m+1)(M+1)}{(m+1)M + (m+2)(M+1)},
$$
and
$$
\max \{ \Theta_{n-1}(x),\Theta_n(x)\} > \tfrac{(m+2)M}{(m+1)M + (m+2)(M+1)}.
$$
\end{itemize}
\end{theorem}

\begin{proof}
From the discussion prior to the statement of the theorem, we see that $\Psi (V_{1,M})\cap \Psi (H_{1,m})$ is a quadrangle with vertices given by the intersection points of the ``top line'' and ``bottom line'' of $\Psi (V_{1,M})$ and the ``left- and right-hand lines'' of $\Psi (H_{1,m})$. Since $V_{1,M}\cap H_{1,m} = \Big[ \tfrac{M}{M+1},\tfrac{M+1}{M+2}\Big) \times \Big[ \tfrac{m}{m+1},\tfrac{m+1}{m+2}\Big)$, an easy calculation yields that these points are given by:
\begin{eqnarray*}
\Psi \left( \tfrac{M+1}{M+2},\tfrac{m+1}{m+2}\right) &=& \left( \tfrac{(m+1)(M+2)}{(m+1)(M+1)+(m+2)(M+2)}, \tfrac{(m+2)(M+1)}{(m+1)(M+1)+(m+2)(M+2)}\right) ,\\
\Psi \left( \tfrac{M+1}{M+2},\tfrac{m}{m+1}\right) &=& \left( \tfrac{m(M+2)}{m(M+1)+(m+1)(M+2)}, \tfrac{(m+1)(M+1)}{m(M+1)+(m+1)(M+2)}\right) ,\\
\Psi \left( \tfrac{M}{M+1},\tfrac{m}{m+1}\right) &=& \left( \tfrac{m(M+1)}{mM+(m+1)(M+1)}, \tfrac{(m+1)M}{mM+(m+1)(M+1)}\right),
\end{eqnarray*}
and $\Psi \left( \tfrac{M}{M+1},\tfrac{m+1}{m+2}\right) = \left( \tfrac{(m+1)(M+1)}{(m+1)M+(m+2)(M+1)}, \tfrac{(m+2)M}{(m+1)M+(m+2)(M+1)}\right)$.
\bigskip\

Case ($i$): In case $1\leq m=M$, we have that the vertices of $\Psi (V_{1,m})\cap \Psi (H_{1,m})$ are given by:
\begin{eqnarray*}
\left( \tfrac{(m+1)(m+2)}{(m+1)^2+(m+2)^2}, \tfrac{(m+1)(m+2)}{(m+1)^2+(m+2)^2}\right), &&\left( \tfrac{m(m+2)}{m(m+1)+(m+1)(m+2)}, \tfrac{(m+1)^2}{m(m+1)+(m+1)(m+2)}\right) ,\\
\left( \tfrac{(m+1)^2}{m(m+1)+(m+1)(m+2)}, \tfrac{m(m+2)}{m(m+1)+(m+1)(m+2)}\right) , && \text{and}\,\, \left( \tfrac{m(m+1)}{m^2+(m+1)^2}, \tfrac{m(m+1)}{m^2+(m+1)^2}\right).
\end{eqnarray*}
Since the slope of the ``top line'' of $\Psi (V_{1,m})$ is $-\frac{(m+1)^2}{(m+2)^2}$, and the slope of the ``right-hand line'' of $\Psi (H_{1,m})$ is $-\frac{(m+2)^2}{(m+1)^2}$, and it follows that for $(\Theta_{n-1}(x),\Theta_n(x))\in \Psi (V_{1,m})\cap \Psi (H_{1,m})$ the minimum of $\Theta_{n-1}(x)$ and $\Theta_n(x)$ is determined by the vertex $\left( \tfrac{(m+1)(m+2)}{(m+1)^2+(m+2)^2}, \tfrac{(m+1)(m+2)}{(m+1)^2+(m+2)^2}\right)$ on the diagonal $\beta =\alpha$:
$$
\min \{ \Theta_{n-1}(x),\Theta_n(x)\} \leq \frac{(m+1)(m+2)}{(m+1)^2+(m+2)^2} < \tfrac{1}{2}.
$$
In a similar way, we see that the vertex $\left( \frac{m(m+1)}{m^2+(m+1)^2}, \frac{m(m+1)}{m^2+(m+1)^2}\right)$ on the diagonal $\beta =\alpha$ yields that:
$$
\max \{ \Theta_{n-1}(x),\Theta_n(x)\} \geq \frac{m(m+1)}{m^2+(m+1)^2} \geq \tfrac{2}{5};
$$
see also Figure~\ref{figure3}.\smallskip\

Case ($ii$): In case $1\leq m<M$, again due to symmetry we only need to consider $\Psi (V_{1,M})\cap \Psi (H_{1,m})$, as $\Psi (V_{1,m})\cap \Psi (H_{1,M})$ is the reflection of $\Psi (V_{1,M})\cap \Psi (H_{1,m})$ in the diagonal $\beta = \alpha$. Note that $\Psi (V_{1,M})\cap \Psi (H_{1,m})$ is ``above'' the diagonal $\beta = \alpha$ (if $M=m+1$, the intersection $\Psi (V_{1,M})\cap \Psi (H_{1,m})$ ``touches'' the diagonal at one of the vertices; see also Figure~\ref{figure3}). Thus, the minimum of $\Theta_{n-1}(x)$ and $\Theta_n(x)$ on $\Psi (V_{1,M})\cap \Psi (H_{1,m})$ is lesser than, or equal to the largest possible value of $\Theta_{n-1}(x)$ on $\Psi (V_{1,M})\cap \Psi (H_{1,m})$, and we conclude that
$$
\min \{ \Theta_{n-1}(x),\Theta_n(x)\} \leq \frac{(m+1)(M+1)}{(m+1)M+(m+2)(M+1)}.
$$
Similarly, the maximum of $\Theta_{n-1}(x)$ and $\Theta_n(x)$ is greater than, or equal to the smallest possible value of $\Theta_n(x)$ on $\Psi (V_{1,M})\cap \Psi (H_{1,m})$, and thus we find that:
$$
\max \{ \Theta_{n-1}(x),\Theta_n(x)\} \geq \frac{(m+2)M}{(m+1)M+(m+2)(M+1)}.
$$
So in case $M=a_{n+2}>a_{n-1}=m$, we see that both values are determined by $\Psi \left( \tfrac{M}{M+1},\tfrac{m+1}{m+2}\right)$.
\end{proof}

\section{Some remarks on sharpening Borel's result}\label{section4}
Note that, it follows from~\eqref{improvementBorel} that whenever $a_{n+1}\geq 2,$ there is already a significant improvement over Borel's one already has a considerable improvement over Borel's result~\eqref{Borel}, and this was also remarked in~\cite{[HN]}. The constant $1/\sqrt{5}$ is replaced by a constant that is at most $1/\sqrt{8}$. Therefore, we can assume that $a_{n+1}=1$. The method we used in Section~\ref{section3} can also be applied in this  case, but now the situation is slightly more complicated. We will return to this in future.

Apart from the RCF, there is a bewildering  number of other continued fraction algorithms. For example, one has the aforementioned $\alpha$-expansions of Nakada (\cite{[N]}), which includes the classical \emph{nearest integer continued fraction expansion} and the \emph{singular continued fraction expansion} as examples. Additionally, there is \emph{Minkowski's Diagonal continued fraction expansion} (\cite{[Kr]}) and Bosma's \emph{optimal continued fraction expansion} (OCF, \cite{[B]}). By definition, for every $x$ each approximation coefficient is smaller than $\tfrac{1}{2}$, and this also holds for the OCF (see also~\cite{[Kr]}). In fact, for the OCF, the minimum of two consecutive approximation coefficients is smaller than $1/\sqrt{5}$. However, for classical continued fraction algorithms such as the NICF the situation is less favorable; as was shown by Tong in 1992 (\cite{[T2]}). Apart from these continued fraction algorithms related to $\text{SL}_2(\Z),$ there are also continued fraction algorithms related to other groups, such as the \emph{Rosen fractions}, which are linked to Hecke groups. It would be interesting to explore whether the approach of Han\u{c}l and his co-authors, as well as the approach of the present paper can be applied to these continued fraction algorithms.

\section*{Acknowledgement}
The research of the first author is supported by the research project ‘Dynamics and Information Research Institute - Quantum Information, Quantum Technologies’ within the agreement between UniCredit Bank and Scuola Normale Superiore (SNS). The first author also acknowledges the support of the Centro di Ricerca Matematica Ennio de Giorgi, SNS and the Institute of Mathematics of the Polish Academy of Sciences for providing excellent working conditions and research travel funds.  The second author acknowledges the hospitality of the Institute of Mathematics of the Polish Academy of Sciences.
The authors would like to thank  the
Erwin Schr\"odinger International Institute for Mathematical Physics in Vienna, where part of this work was completed, for
their hospitality.

\Addresses
\end{document}